\documentclass{article}
\usepackage{amsthm}
\usepackage{amssymb}
\usepackage{amsmath}
\usepackage{cancel}
\usepackage{xcolor}
\usepackage{graphicx}
\newtheorem{theorem}{Theorem}

\newtheorem{lemma}[theorem]{Lemma}
\newtheorem{example}[theorem]{Example}

\newtheorem{definition}[theorem]{Definition}

\title{Linear non-homogenous patterns and prime power generators in numerical semigroups associated to combinatorial configurations}
\author{Klara Stokes and Maria Bras-Amor\'os}
\begin{document}
\maketitle
\begin{abstract}
It is proved that the numerical semigroups associated to the combinatorial configurations satisfy a family of non-linear symmetric patterns. 
Also, these numerical semigroups are studied for two particular classes of combinatorial configurations.  
\end{abstract}
\section{Introduction}
In this article we will discuss some properties and examples of numerical semigroups associated to the existence of combinatorial configurations. 
The link between these two objects was presented in \cite{BrasStokes}. 
We will now introduce the concepts that will be used in the rest of this article.

\subsection{Combinatorial configurations}
An \emph{incidence structure} is a set of \emph{points} $\mathcal{P}$ and a set of \emph{lines} $\mathcal{L}$, together with an incidence relation between these two sets. 
If a point $p$ and a line $l$  are incident, then we say that $l$  \emph{goes through} $p$, that $p$ \emph{is on} $l$, and so on.   
We say that a pair of lines that goes through the same point $p$ \emph{meet} or \emph{intersect} in $p$.  

A \emph{combinatorial configuration} is an incidence structure in which there are $r$ lines through every point, $k$ points on every line and such that through any pair of points there is at most one line.  The last condition can be replaced by requiring any pair of lines to meet in at most one point. 
A general reference for combinatorial configurations is \cite{Gropp07} and the book \cite{Grunbaum} collects many results on combinatorial configurations, 
although it focuses on geometrically realizable configurations. 

We will use the notation \emph{$(v,b,r,k)$-configuration} to refer to a combinatorial configuration 
with $v$ points, $b$ lines, $r$ lines through every point and  $k$ points on every line. 
When $v$ and $b$ are not known or not important, then we use the notation \emph{$(r,k)$-configuration.}
We say that a combinatorial configuration is \emph{balanced} if $r=k$. This implies that $v=b$. 
Figure~\ref{fig:1} shows some examples of combinatorial configurations. 

\begin{figure}
\begin{tabular}{ccc}
\includegraphics[width=0.23\textwidth]{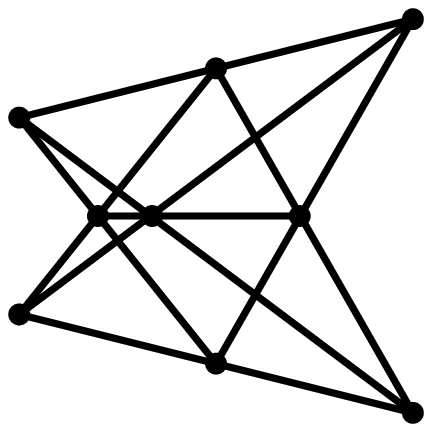}&
\includegraphics[width=0.26\textwidth]{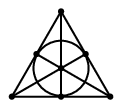}&
\includegraphics[width=0.31\textwidth]{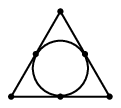}\\
Pappus' configuration&The Fano plane&A non-balanced \\
&&configuration\\
$(v,b,r,k)=(9,9,3,3)$&$(v,b,r,k)=(7,7,3,3)$&$(v,b,r,k)=(6,4,2,3)$
\end{tabular}
\caption{Examples of combinatorial configurations}
\label{fig:1}
\end{figure}
The following results are well-known. We include the simple proofs for the sake of completeness. 
\begin{lemma}\label{thm:neccond}
\begin{enumerate}
\item \label{neccond1} $v\geq r(k-1)+1$ and $b\geq k(r-1)+1$;
\item \label{neccond2} $vr=bk$. 
\end{enumerate}
\end{lemma}
\begin{proof}
\begin{enumerate}
\item Take a point $p$. There are $r$ lines through $p$ with $k-1$ more points, hence at least $r(k-1)+1$ points. The other inequality is proved analogously. 
\item There are $v$ points in $r$ incidence relations and $b$ lines in $k$ incidence relations. Since the incidence relation is symmetric we get $vr=bk$. 
\end{enumerate}
\end{proof}




A natural question to ask is for which  parameter sets do combinatorial $(v,b,r,k)$-configurations exist.
Actually, the four parameters $(v,b,r,k)$ are redundant, and we only need the three parameters
$$(d,r,k)$$ 
with $d=\frac{v\gcd(r,k)}{k}=\frac{b\gcd(r,k)}{r}\in \mathbb{Z}$. 
Indeed, we have seen that in a combinatorial configuration necessarily $vr=bk$.
Therefore the number of points $v$ and the number of lines $b$ is given by
$$v=\frac{bk}{r}=d\frac{k}{\gcd(r,k)}$$ 
and symmetrically
$$b=\frac{vr}{k}=d\frac{r}{\gcd(r,k)}.$$ 
We associate the integer $d$ to the configuration. 
If one prefers, one can also express this integer as $$d=\frac{vr}{\mbox{lcm}(r,k)}=\frac{bk}{\mbox{lcm}(r,k)}.$$

\begin{definition}
For $r,k\in\mathbb{N}$, $r,k\geq 2$ we define \medskip\\
$S_{(r,k)}:=\{d\in\mathbb{N}:\exists \mbox{ combinatorial }(v,b,r,k)\mbox{-configuration and }$\begin{flushright}
$~~~~~~~~~~~~~~~~~~~~~~~~~~~~~~~~~~~~~~~~~~~~~~v=d\frac{k}{\gcd(r,k)}, b=d\frac{r}{\gcd(r,k)}\}.$\end{flushright}
\end{definition}
The set $S_{(r,k)}$ is the object of study in this article. 

\subsection{Numerical semigroups}

A \emph{numerical semigroup} is a subset $S\subseteq\mathbb{N}\cup \{0\}$, 
such that $S$ is closed under addition,  $0\in S$ and the complement $(\mathbb{N}\cup\{0\})\setminus S$ is finite. 
The\textit{ gaps} of a numerical semigroup are the elements in the complement of the numerical semigroup and the \emph{genus} of a numerical semigroup is the number of gaps of the numerical semigroup. 
The \textit{multiplicity} of a numerical semigroup is its smallest non-zero element.
Every numerical semigroup has a minimal set of \emph{generators}. 
The \textit{conductor} of a numerical semigroup is the smallest element such that all subsequent natural numbers belong to the numerical semigroup. 
If the numerical semigroup is generated by the two elements $a$ and $b$, then the conductor $c$ is given by 
\begin{equation}
c=(a-1)(b-1).
\label{eqcond}
\end{equation} 
In general,  we do not have an explicit expression of the conductor in terms of the generators. 
However, the conductor can be bounded as a function of other properties of the numerical semigroup. 
For example, we have the following upper bound  in terms of the genus (see Lemma 2.14 in \cite{RoGa}). 

\begin{lemma}\label{1:fita:conductorboundgenus}
The genus $g$ and the conductor $c$ of a numerical semigroup always satisfy
$$2g\geq c.$$
\end{lemma}

\begin{example}
$$\langle 3,7\rangle=\{0,3,6,7,9,10,12,13,14,15,16,\dots\}$$ is the numerical semigroup generated by 3 and 7. 
In this numerical semigroup the multiplicity is 3, the conductor is 12 and the gaps are $\{1,2,4,5,8,11\}$, so that the genus is 6.
\end{example}
For a general reference on numerical semigroups see \cite{RoGa}. 
The link between numerical semigroups and combinatorial configurations is to be found in the following result from \cite{BrasStokes}. 
\begin{theorem}\label{main}
For every pair of integers $r,k\geq 2$, $S_{(r,k)}$ is a numerical semigroup.
\end{theorem}

\section{Example: numerical semigroups associated to balanced configurations}
For balanced configurations, i.e. when $r=k$, the number of points $v$ equals the number of lines $b$, and also the associated integer $d$, so that $d=v=b$. 
In \cite{Grunbaum}, two ways to combine two balanced combinatorial configurations in order to construct a single larger one are described. 
Both constructions let the point set of the new configuration consist of the union of the point sets of the two original configurations, 
with the exception of removing or adding one point, respectively.
Since the number of points equals the associated integer, we get that 
$$d_1,d_2\in S_{(r,r)} \Rightarrow d_1+d_2-1$$
and 
$$d_1,d_2\in S_{(r,r)} \Rightarrow d_1+d_2+1.$$
As a consequence, given an element $d\in S_{(r,r)}$ we have that $2d-1,2d,2d+1\in S_{(r,r)}$. 
In particular, since $d$ and $2d-1$ are coprime, given a non-zero element in $S_{(r,r)}$, this is enough to prove that the complement of $S_{r,r}$ in $\mathbb{N}\cup \{0\}$  is finite. 

From results on the existence of cyclic configurations, difference sets and Golomb rulers, 
the existence of $(d,r,r)$-configurations can be confirmed for many values of $d$.  
Some results on non-existence can be deduced as a consequence of a Theorem of Bose and Connor.
In the table below, the numbers that appear are confirmed to belong to the numerical semigroup, the crossed-out numbers are confirmed to not belong to the numerical semigroup and the numbers within question-marks  are not confirmed. 
$$\begin{array}{|c|c|ccccccccc|}
\hline
r=k&\pi&S_{(r,k)}&\setminus\{0\}&&&&&&&\\
\hline
3&7&7&\rightarrow&&&&&&&\\
4&13&13&\rightarrow&&&&&&&\\
5&21&21&\cancel{22}&23&\rightarrow&&&&&\\
6&31&31&\cancel{32}&\cancel{33}&34&\rightarrow&&&&\\
7&43&\cancel{43}&\cancel{44}&45&?46?&?47?&48&\rightarrow&&\\
8&57&57&\cancel{58}&?59?&?60?&?61?&?62?&63&\rightarrow&\\
9&73&73&\cancel{74}&?75?&?76?&?77?&78&?79?&80\rightarrow&\\
\hline
\end{array}$$
From Lemma~\ref{thm:neccond} it is easy to see that a lower bound for the multiplicity $m$ of the numerical semigroup $S_{(r,r)}$ is $m\geq r^2-r+1$. 
This bound is attained if and only if there exists a finite projective plane of order $r-1$. 
We denote $P(r)=r^2-r+1$.

A \emph{Golomb ruler} $G_r$ of order $r$ is an ordered set of $r$ integers $a_1, a_2,\dots,a_r$ such that
$0 \leq a_1 < a_2 <\dots < a_r$ and all the differences $\{a_i-a_j : 1 \leq j < i \leq r\}$ are distinct.
The length $L_G(r)$ of the ruler $G_r$ is equal to $a_r-a_1$. 
We denote by $L_{\bar{G}}(r)$ the length of the shortest known Golomb ruler of order $r$.

It can be proved that for all $v$ such that $v \geq  2L_{\bar{G}}(r) + 1$, there exists a (cyclic) balanced combinatorial configuration with parameters $(v,v,r,r)$ (see \cite{Gropp}).
Therefore an upper bound of the conductor $c$ of $S_{(r,r)}$ is  $c\leq 2L_{\bar{G}}(r) + 1$. 
Following \cite{Davydov} we call this the Golomb bound and denote it by  $G(r)= 2L_{\bar{G}}(r) + 1$. 
It is obvious that $P(r)\leq G(r)$. 
According to \cite{Davydov}, for $r\in[10,\dots,37]$,  for a percentatge of between $35\%$ ($r=10$) and $89\%$ ($r=16$) of the integers between $P(r)$ and $G(r)$, it is known whether they belong to $S_{(r,r)}$ or not.


%


\section{Numerical semigroups associated to configurations with coprime parameters}
We can always construct an $(r,k)$-configuration, for any choice of parameters $(r,k)$.
Indeed, let $q\geq \max(r,k)$ and let $AG(2,q)$ be the finite affine plane over the finite field with $q$ elements. 
It has $q^2$ points and $q^2+q$ lines. 
There are $q$ points on every line and $q+1$ lines go through every point. 

We say that two lines are parallel if they do not intersect in any point. 
The lines in $AG(2,q)$ can be partitioned into $q+1$ classes of parallel lines of $q$ lines each, so that for every point there is exactly one line from every class that goes through that point. 

Consider the incidence structure constructed by taking the lines from $r$ parallel classes of $AG(2,q)$ and restrict these to the points located on $k$ of the lines of an additional parallel class of lines. 
It is easy to see that this incidence structure is an $(r,k)$-configuration and that it has $kq$ points and $rq$ lines. 
The associated integer to this $(r,k)$-configuration is therefore $kq\frac{\gcd(r,k)}{k}=q\gcd(r,k)$. 

This construction works whenever $q\geq \max(r,k)$. 
Finite affine planes are known to exist if $q$ is a prime power. 
Indeed, there is a finite affine plane for every finite field. 
There are also other finite affine planes, for non-prime orders. 
However, it is not known if there exist finite affine planes of order that is not a prime power.
As a consequence, we get the following result. 

\begin{lemma}
If $\gcd(r,k)=1$, then any prime power $q\geq \max(r,k)$ belongs to $S_{(r,k)}$. 
\end{lemma}
Numerical semigroups that are generated by prime powers have according to our knowledge not been previously treated in the literature. 
We present the following upper bounds for this type of numerical semigroups. 
\begin{theorem}\label{thm9}
Let $c$ be the conductor of a numerical semigroup that contains all prime powers larger than or equal to a given integer $n$. 
Then this conductor satisfies
$$c\leq 2\prod_{p~prime,~p<n}(\lfloor\log_p(n-1)\rfloor+1).$$
\end{theorem}
\begin{proof}
In Lemma~\ref{1:fita:conductorboundgenus} we saw that the conductor of a numerical semigroup is smaller or equal to two times the genus. 

Suppose that $\Lambda$ is a numerical semigroup that contains all prime powers larger
than or equal to a given integer $n$. 
We want to estimate the genus of $\Lambda$.
Then any gap $x$ can be expressed as a
product 
$$x=p_1^{n_1}\cdots p_k^{n_k}$$
with $n_i$ integers such that $1\leq n_i\leq \log_{p_i}(n-1)$ for all $i$. 
In particular $p_1,\dots,p_k$ are prime numbers smaller than $n$. 

Indeed, decompose $x$ as a product of powers of different primes $x = p_1^{n_1}\cdots p_k^{n_k}$. 
If $n_i > \log_{p_i}(n-1)$ for some $i$ then $p_i^{n_i}$ is a prime power larger than or
equal to $n$ and so it belongs to $\Lambda$ and so does any multiple of it, like $x$.

Therefore the genus, that is, the number of gaps of $\Lambda$, is at most 
$$\prod_{p~prime,~p<n}(\lfloor\log_{p}(n-1)\rfloor+1),$$
so that the conductor of $\Lambda$ is at most 
$$2\prod_{p~prime,~p<n}(\lfloor\log_{p}(n-1)\rfloor+1).$$

\end{proof}

\section{Linear non-homogeneous patterns} 

A pattern of length $n$ admitted by a numerical semigroup $S$ is a polynomial  $p(X_1,\dots,X_n)$ with non-zero integer coefficients,  such that, for every ordered sequence of $n$ elements $s_1\geq \dots\geq s_n$ from $S$, we have $p(s_1,s_2,\dots,s_n)\in S$. 
\begin{example}
Let $S$ be a numerical semigroup such that for every triple $s_1\geq s_2 \geq s_3$ in $S$ we have $s_1+s_2-s_3\in S$. Then the polynomial $X_1+X_2-X_3$ is a pattern for $S$.  
\end{example} 
A pattern is called linear, homogenous or symmetric  if the pattern polynomial is linear, homogenous or symmetric. 
Linear, homogenous patterns were first introduced and studied in \cite{MariaPedro}. 
Linear, non-homogeneous patterns have recently been studied in \cite{AlbertMaria}. 

\begin{theorem}\label{thmpatterns}
Let $S_{(r,k)}$ be a numerical semigroup associated to the $(r,k)$-configurations. 
Then $S_{(r,k)}$ admits the pattern  $$X_1+X_2-n$$ for all $n\in [1,\dots,\gcd(r,k)]$. 
\end{theorem}

\begin{proof}
Take two $(r,k)$-configurations $A$ and $B$ with associated integers $d_A$ and $d_B$. 
Then $A$ has $v_A=d_A\frac{k}{\gcd(r,k)}$ points and $b_A=d_A\frac{r}{\gcd(r,k)}$ lines, while $B$ has $v_B=d_B\frac{k}{\gcd(r,k)}$ points and $b_B=d_B\frac{r}{\gcd(r,k)}$ lines. 

Remove $a:=nk/\gcd(r,k)$ points $p_1,\dots,p_a$  on a line $L$ in $A$. 
Also remove $b:=nr/\gcd(r,k)$ lines $l_1,\dots,l_b$ through a point $p$ in $B$. 
The line $L$ is now missing $nk/\gcd(r,k)$ points. 
The $\left(nk/\gcd(r,k)\right)(r-1)$ lines that previously went through the removed points are now missing one point each. 

Replace the missing points on $L$ with $p$ together with $nk/\gcd(r,k)-1$ other points that previously were on $l_1$.  
There are now missing $nr/\gcd(r,k)-1$ lines through $p$. 
Replace these by letting the lines that previously went through $p_1$ now go through $p$. 
Replace the rest of the missing points on the lines which previously went through $p_1,\dots,p_a$  with the points in $B$ that previously were on  $l_1,\dots,l_b$, until there are $r$ lines going through all these points. 

It is easy to check that the resulting incidence structure is an $(r,k)$-configuration with 
$$v=v_A+v_B-a=d_A\frac{k}{\gcd(r,k)}+d_B\frac{k}{\gcd(r,k)}-\frac{nk}{\gcd(r,k)}=(d_A+d_B-n)\frac{k}{\gcd(r,k)},$$
so that the  associated integer is $d_A+d_B-n$. 
\end{proof}
Figure~\ref{fig:2} shows an example of the construction used in Theorem~\ref{thmpatterns} for $r=3$ and $k=5$. In this case $\gcd(r,k)=1$, so there is only one choice, $n=1$.
\begin{figure}
\begin{center}
  \includegraphics[width=0.90\textwidth]{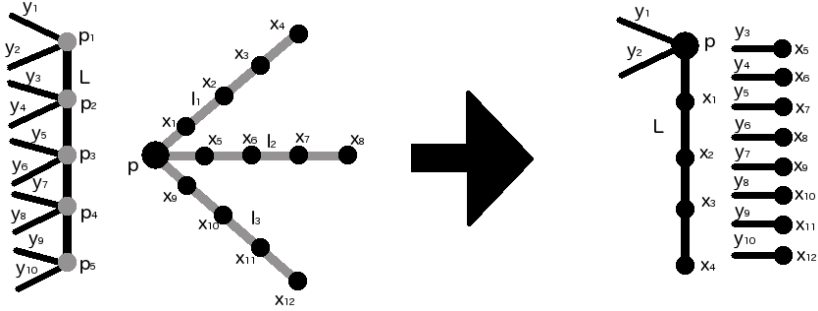}
\end{center}
\caption{Theorem~\ref{thmpatterns} for $r=3$, $k=5$ and $n=1$. The grey points and lines in the two combinatorial configurations on the left are removed and the resulting configuration is shown on the right.}
\label{fig:2}
\end{figure} 

Observe that the patterns in Theorem~\ref{thmpatterns} are linear and non-homogeneous, and also that they  are symmetric.
From the existence of these patterns, the following upper bound for the conductor of $S_{(r,k)}$ can be obtained. 



\begin{theorem}\label{thm12}
The conductor $c$ of a numerical semigroup $S_{(r,k)}$ associated to the $(r,k)$-configurations is bounded by
$$c\leq (x+1)m-x\gcd(r,k)$$
where $m$ is the multiplicity of $S_{(r,k)}$ and 
$x=\left\lfloor\frac{m-2}{\gcd(r,k)}\right\rfloor$.
\end{theorem}
\begin{proof}
If $d\in S_{(r,k)}$ then $2d-n\in S_{(r,k)}$ for $n\in [1,\gcd(r,k)]$.
Therefore  the intervals $I_x=[(x+1)d-x\gcd(r,k)),(x+1)d]$ belong to $S_{(r,k)}$ for $d\in S_{(r,k)}$. 
If there is a gap between $I_x$ and $I_{x-1}$, then $(x+1)d-x\gcd(r,k)>xd+1$, so that $x<\frac{d-1}{\gcd(r,k)}$. 
Hence, the largest $x$ for which $I_x$ and $I_{x-1}$ are separated by at least one gap is at most $\left\lfloor\frac{d-2}{\gcd(r,k)}\right\rfloor$. 
So, the conductor must be at least the first element of $I_x$ with $x= \left\lfloor\frac{d-2}{\gcd(r,k)}\right\rfloor$ and this is exactly $(x + 1)d − x gcd(r, k))$. 
The minimality of the multiplicity suggests then substituting $d$ by $m$.
\end{proof}

In order to compare the different bounds that have been presented in this article, we give the values for these bounds for balanced configurations of small parameters in Table~\ref{tab1} and configurations with small coprime parameters in Table~\ref{tab2}. 
From Table~\ref{tab1} it is clear that the bound from Theorem~\ref{thm12} in this case is far from being sharp. 
The calculations of the bound from Theorem~\ref{thm12} were seeded with the upper bound for the multiplicity $q\gcd(r,k)$, where $q$ is the smallest prime power larger than $\max(r,k)$. 
If the real multiplicity is used, the bound from Theorem~\ref{thm12} will give better results. 
\begin{figure}
\begin{tabular}{|c|c|c|c|}
\hline
r&P(r)&G(r)&Theorem~\ref{thm12}\\
\hline
3 & 7  & 7  & 21 \\
4 & 13 & 13 & 52 \\
5 & 21 & 23 & 105\\
6 & 31 & 35 & 258\\
7 & 43 & 48 & 301\\
8 & 57 & 63 & 456\\
9 & 73 & 80 & 657\\
\hline
\end{tabular}
\caption{Bounds for balanced $(r,r)$-configurations}
\label{tab1}
\end{figure}
\begin{figure}
\begin{tabular}{|c|c|c|c|}
\hline
r&k&Theorem~\ref{thm9}&Theorem~\ref{thm12}\\
\hline
3&4&8&10\\
3&5&12&17\\
3&7&24&37\\
3&8&48&50\\
3&10&96&101\\
3&11&96&101\\
3&13&192&145\\
4&5&12&17\\
4&7&24&37\\
4&9&64&65\\
4&11&96&101\\
4&13&192&145\\
5&6&24&37\\
5&7&24&37\\
5&8&48&50\\
5&9&64&65\\
5&11&96&101\\
5&12&192&145\\
5&13&192&145\\
5&14&384&226\\
\hline
\end{tabular}
\caption{Bounds for $(r,k)$-configurations with $\gcd(r,k)=1$}
\label{tab2}
\end{figure}


\section*{Conclusions}
We have proved that the numerical semigroups attached to the existence of combinatorial configurations allow a family of linear, symmetric and non-homogenous patterns. We have also studied the numerical semigroups attached to the balanced combinatorial $(r,k)$-configurations ($r=k$) and to the combinatorial $(r,k)$-configurations for which $r$ and $k$ are coprime.

\section*{Acknowledgements}
The authors would like to thank an anonymous referee for the idea to construct combinatorial configurations from finite affine planes. 
Partial support by the Spanish MEC projects ARES
(CONSOLIDER INGENIO 2010 CSD2007-00004), RIPUP
(TIN2009-11689) and ICWT (TIN2012-32757), is acknowledged. 
The second author is with the UNESCO Chair in Data Privacy, 
but her views do not necessarily reflect those of UNESCO, nor commit that organization.  
\bibliographystyle{plain}

\end{document}